\documentclass[11pt]{article}

\usepackage[english]{babel}
\usepackage{subfigure}
\usepackage[latin1]{inputenc}
\usepackage{verbatim}
\usepackage{psfrag}
\usepackage{mathrsfs}
\usepackage{bm}
\usepackage{amsmath}
\usepackage{amssymb}
\usepackage{amsfonts}
\usepackage{amsthm}
\usepackage{float}
\usepackage{setspace}
\usepackage{enumerate}
\usepackage{scrextend}

\newtheorem{thm}{Theorem}[section]

\newtheorem{lemma}[thm]{Lemma}

\newtheorem{claim}{Claim}[thm]

\theoremstyle{definition}

\title{Enumeration of extensions of the\\ cycle matroid of a complete graph}
\date{\today}

\author{Peter Nelson, Shayla Redlin, and Jorn van der Pol\\
\small Department of Combinatorics and Optimization\\[-0.8ex]
\small University of Waterloo\\[-0.8ex] 
\small Waterloo, Ontario, Canada\\}

\begin{document}

\maketitle

\begin{abstract}
    We prove that the number of single element extensions of $M(K_{n+1})$ is $2^{{n\choose n/2}(1+o(1))}$.
    This is done using a characterization of extensions as ``linear subclasses''.
\end{abstract}

\section{Introduction}
\label{sec:intro}

A matroid $N$ is an \textit{extension} of a matroid $M$ if $N\backslash e = M$ for some element $e$ of $N$ where $e$ is not a coloop. 
Similarly, a matroid $N$ is a \textit{coextension} of a matroid $M$ if $N/e = M$ for some element $e$ of $N$ where $e$ is not a loop.
Our goal is to enumerate the extensions and coextensions of $M(K_{n+1})$, the cycle matroid of a complete graph on $n+1$ vertices. In this paper, we prove the following result about the number of extensions.\footnote{In this paper, $o(1)$ denotes a function of $n$ which goes to $0$ as $n$ goes to infinity and we say ${n\choose x} := {n\choose \lfloor x \rfloor}$.}

\begin{thm}
\label{thm:ext_clique}
The number of extensions of $M(K_{n+1})$ is $2^{{n\choose n/2}(1+o(1))}$.
\end{thm}

It is well known that coextensions of graphic matroids correspond to biased graphs, 
which are independently well-studied \cite{zaslavsky89, Zaslavsky, zaslavsky95}.
In \cite{peter_jorn}, Nelson and Van der Pol proved that the number of biased graphs on the complete graph $K_{n+1}$ is $2^{\frac{1}{2}n!(1+o(1))}$, which implies the following theorem.

\begin{thm}[\cite{peter_jorn}]
\label{thm:coext_clique}
The number of coextensions of $M(K_{n+1})$ is $2^{\frac{1}{2}n!(1+o(1))}$.
\end{thm}

The proof of Theorem \ref{thm:ext_clique} has a structure similar to the proof of Theorem~\ref{thm:coext_clique}. The strategy begins by finding a natural lower bound for the number of extensions or coextensions. This is done by finding a large set of objects that describes an extension or coextension such that each subset also describes an extension or coextension. An upper bound is then determined by viewing the large sets that describe extensions or coextensions as stable sets in an auxiliary graph and applying the container method. 
The Boolean lattice is used as the auxiliary graph in this paper, so known results about enumerating antichains in a Boolean lattice are used instead of a direct application of the container method.

Cycle matroids of complete graphs, which we sometimes refer to as cliques, are highly symmetric and dense, so we might expect the collection of extensions and coextensions to be tame and well-understood.
However, Theorems \ref{thm:ext_clique} and \ref{thm:coext_clique} imply that the number of matroids that are one contraction or deletion away from being a clique on $N$ elements is ``close'' to~$2^{2^{\sqrt{N}}}$.
In comparison, the number of extensions of a projective geometry with $N$ elements is ``close'' to $2^{\log ^2 N}$.\footnote{In this paper, $\log$ denotes the base-2 logarithm.} This implies that the number of extensions and coextensions of a clique is much larger than the number of extensions of some other classes of matroids.

For $N \in \mathbb{N}$, let $m(N)$ denote the number of matroids on $N$ elements.
Knuth \cite{knuth_74} found a large family of matroids on $N$ elements in 1974 which implies the lower bound $\log\log m(N) \ge N(1-o(1))$.
This combined with the trivial upper bound of $\log\log m(N) \le N$ implies that $m(N)$ is ``close'' to $2^{2^N}$.
Thus, the number of extensions and coextensions of a clique on $N$ elements is close to the total number of matroids on $N$ elements in the sense that they are both doubly exponential in $N$.

The rest of the paper is organized as follows. In Section \ref{sec:prelim}, we set up the proof of Theorem \ref{thm:ext_clique} with some preliminary definitions and lemmas.
The proof itself is then provided in Section \ref{sec:proof}.

\section{Preliminaries}
\label{sec:prelim}

In this paper, we follow the definitions and notation of Oxley \cite{oxley}. We begin by reviewing a few relevant definitions.
For a matroid $M$, let $\mathcal{H}(M)$ denote the set of hyperplanes of $M$.
A \textit{linear subclass} of a matroid $M$ is a subset $\mathcal{H'}\subseteq \mathcal{H}(M)$ such that if two hyperplanes $F,F' \in \mathcal{H}'$ intersect in a flat $S$ of rank $r(M)-2$, then all hyperplanes $F''$ that contain $S$ are in $\mathcal{H}'$ as well.
In~\cite{crapo}, Crapo proved that the linear subclasses of a matroid $M$ parameterize the extensions of $M$.

For an integer $n$, let $[n] = \{1,2,\dots,n\}$ and let $K_{n}$ be the complete graph on vertex set $[n]$.
Recall that the clique $M(K_{n})$ is the cycle matroid of $K_{n}$.
Our goal is to enumerate the extensions of $M(K_{n+1})$. By using Crapo's parameterization, it suffices to enumerate the linear subclasses of $M(K_{n+1})$. 
Note that the rank of $M(K_{n+1})$ is $n$.
Let $\mathcal{E}(n+1)$ denote the set of extensions of $M(K_{n+1})$ and let $\mathcal{L}(n+1)$ denote the set of linear subclasses of $M(K_{n+1})$.
Let $\mathcal{P}(n)$ denote the power set of $[n]$.

Recall that a biased graph is a pair $(G,\mathcal{B})$ where $G$ is a finite graph and $\mathcal{B}$ is a collection of \textit{balanced} cycles of $G$; that is, a collection of cycles where there do not exist cycles $C_1,C_2,C_3$ in a theta-subgraph of $G$ for which $|\{C_1,C_2,C_3\}\cap \mathcal{B}| = 2$.
This motivates the following definition.
We say a collection of sets $\mathcal{B}\neq \{\emptyset\}$ is \textit{linear} if there do not exist disjoint sets $X,Y$ for which $|\{X,Y,X\cup Y\}\cap \mathcal{B}| = 2$.
Observe that no linear set can contain the empty set: If $\mathcal{B}$ is a linear set and $\emptyset,X\in \mathcal{B}$, then $|\{\emptyset, X, X\cup \emptyset\}\cap \mathcal{B}| = 2$.
Let $\mathcal{Q}(n)$ denote the set of linear subsets of $\mathcal{P}(n)$.

The proof of the following lemma shows that linear subsets of $\mathcal{P}(n)$ correspond to linear subclasses of $M(K_{n+1})$.
In this proof, it is helpful to think of flats of $M(K_{n+1})$ as partitions of vertices. 
We say a \textit{$k$-partition} of a set $S$ is an unordered partition of $S$ into $k$ nonempty parts. 
A set $F$ is a \text{rank-$(n+1-k)$} flat of $M(K_{n+1})$ if and only if the subgraph $([n+1],F)$ has precisely $k$ components, each of which is complete.
Thus, rank-$(n+1-k)$ flats of $M(K_{n+1})$ correspond to $k$-partitions of $[n+1]$.
Flats $F$ and $F'$ satisfy $F \subseteq F'$ if and only if the partition corresponding to $F$ refines the partition for $F'$.

\begin{lemma}
\label{lem:bijection}
$|\mathcal{E}(n+1)| = |\mathcal{Q}(n)|$.
\end{lemma}

\begin{proof}
By Crapo's parameterization of extensions as linear subclasses, we have that $\mathcal{E}(n+1)$ is in bijection with $\mathcal{L}(n+1)$; therefore, it suffices to prove that there exists a bijection between $\mathcal{L}(n+1)$ and $\mathcal{Q}(n)$.
For this proof, let $M$ denote $M(K_{n+1})$.

For each hyperplane $H$ of $M$, let $\psi(H)\subseteq [n]$ be the vertex set of the unique component of $([n+1],H)$ not containing the vertex $(n+1)$.
Since, for each nonempty $X\subseteq [n]$, the partition $\{X,[n+1]\setminus X\}$ gives rise to a hyperplane $H$ of $M$ with $\psi(H) = X$, the function $\psi$ is a bijection from the set of hyperplanes of $M$ to $\mathcal{P}(n)\setminus \{\emptyset\}$.

\begin{claim}
\label{clm:triples}
Let $\mathcal{F}$ be a collection of at least three hyperplanes of $M$ for which $r(\cap \mathcal{F}) = n-2$. Then $|\mathcal{F}|=3$ and there are disjoint nonempty sets $S,T \subseteq [n]$ for which $\psi(\mathcal{F}) = \{S,T,S\cup T\}$.
Moreover, all disjoint nonempty sets $S,T\subseteq [n]$ satisfy $r(\psi^{-1}(S)\cap \psi^{-1}(T)\cap \psi^{-1}(S\cup T)) = n-2$.
\end{claim}

\begin{proof}
The intersection $\cap \mathcal{F}$ is a flat of rank $n-2$, so it corresponds to a $3$-partition of $[n+1]$, say $\{X_0,X_1,X_2\}$.
Since hyperplanes correspond to $2$-partitions of $[n+1]$, a hyperplane that contains $\cap \mathcal{F}$ corresponds to a partition with parts $X_i$ and $X_j\cup X_k$ where $\{i,j,k\} = \{0,1,2\}$. There are three such partitions; thus $|\mathcal{F}| = 3$.
Without loss of generality, we may assume that $n+1 \in X_2$, hence $\psi(\mathcal{F}) = \{X_0,X_1,X_0\cup X_1\}$.

Now suppose $S,T \subseteq [n]$ are disjoint nonempty sets. The partition $\{S,[n+1]\setminus S\}$ corresponds to the hyperplane $\psi^{-1}(S)$. Similarly, we know that $\{T,[n+1]\setminus T\}$ corresponds to $\psi^{-1}(T)$ and $\{S\cup T,[n+1]\setminus (S\cup T)\}$ corresponds to $\psi^{-1}(S\cup T)$.
The coarsest common refinement of the partitions $\{S,[n+1]\setminus S\}$, $\{T,[n+1]\setminus T\}$, and $\{S\cup T,[n+1]\setminus (S\cup T)\}$ is the $3$-partition $\{S, T,[n+1]\setminus (S\cup T)\}$, which corresponds to a flat of rank $n-2$. Thus, it follows that $r(\psi^{-1}(S)\cap \psi^{-1}(T)\cap \psi^{-1}(S\cup T)) = n-2$.
\end{proof}

Since the function $\psi$ gives a bijection between hyperplanes of $M$ and nonempty subsets of $[n]$, it remains to show that a set $\mathcal{H}'$ of hyperplanes of~$M$ is a linear subclass if and only if $\psi(\mathcal{H}')$ is a linear collection of subsets of~$[n]$.
Let $\mathcal{H}'$ be a set of hyperplanes of $M$ and let $\mathcal{S} = \{\psi(H): H\in \mathcal{H}'\}$.

Suppose $\mathcal{S}$ is not linear. Thus, there exists a 3-set $\{X,Y,X\cup Y\}$ where $X$ and $Y$ are disjoint and nonempty such that $|\mathcal{S}\cap \{X,Y,X\cup Y\}| = 2$.
This implies that $\mathcal{H}'\cap \{\psi^{-1}(X),\psi^{-1}(Y),\psi^{-1}(X\cup Y)\} = 2$.
By Claim \ref{clm:triples}, we know that $r(\psi^{-1}(X)\cap \psi^{-1}(Y)\cap \psi^{-1}(X\cup Y)) = n-2$.
Thus, exactly two hyperplanes in $\mathcal{H}'$ contain the rank-$(n-2)$ flat $\psi^{-1}(X)\cap \psi^{-1}(Y)\cap \psi^{-1}(X\cup Y)$, so $\mathcal{H}'$ is not a linear subclass.

Now, suppose $\mathcal{H}'$ is not a linear subclass. Thus, there exist hyperplanes $H_0,H_1,H_2$ that contain a flat $F$ of rank $n-2$ such that $H_0,H_1\in \mathcal{H}'$ and $H_2\notin \mathcal{H}'$. By Claim \ref{clm:triples}, we know that $\{\psi(H_0),\psi(H_1),\psi(H_2)\} = \{X, Y, X\cup Y\}$ where $X$ and $Y$ are nonempty disjoint subsets of $[n]$.
This implies that $|\mathcal{S}\cap \{X,Y,X\cup Y\}| = 2$, so $\mathcal{S}$ is not linear.
\end{proof}

A linear set $\mathcal{B}$ is \textit{scarce} if $|\{X,Y,X\cup Y\}\cap\mathcal{B}| \le 1$ for all disjoint sets $X,Y$.
Let $\mathcal{Q}_s(n)$ denote the set of scarce linear subsets of $\mathcal{P}(n)$.
Since $\mathcal{Q}_s(n)$ is a subset of $\mathcal{Q}(n)$, Lemma \ref{lem:bijection} implies that the number of scarce linear subsets of $\mathcal{P}(n)$ is a lower bound for the number of extensions of $M(K_{n+1})$.

\section{The Proof}
\label{sec:proof}

In the previous section, we found a representation of hyperplanes of a clique $M(K_{n+1})$ as nonempty subsets of $[n]$, which was used to show a bijection exists between extensions and linear subsets of $\mathcal{P}(n)$.
In this section, we prove Theorem \ref{thm:ext_clique} by first counting the number of scarce linear subsets of $\mathcal{P}(n)$ by comparing them to antichains in the Boolean lattice.
This will give us a lower bound for the number of extensions.
To complete the proof, we show that $|\mathcal{Q}(n)|$, the number of linear subsets, is not ``much more'' than $|\mathcal{Q}_s(n)|$, the number of scarce linear subsets, thereby giving us an upper bound for the number of extensions.

Consider the poset on $\mathcal{P}(n)$ partially ordered by the subset relation, also referred to as the Boolean lattice and denoted $\mathcal{P}(n)$.
An \textit{antichain} of $\mathcal{P}(n)$ is a set of elements $A$ such that no two elements in $A$ are comparable. An \textit{intersecting antichain} is an antichain $A$ where $X \cap Y \neq \emptyset$ for all $X,Y\in A$. That is, a collection $A$ of sets is an intersecting antichain if and only if no two sets $X,Y$ in $A$ are comparable or disjoint.
Let $\mathcal{A}_I(n)$ denote the set of intersecting antichains containing nonempty elements of $\mathcal{P}(n)$.

\begin{lemma}
\label{lem:stable_sets}
    $\mathcal{Q}_s(n)=\mathcal{A}_I(n)$.
\end{lemma}

\begin{proof}
Scarce linear sets in $\mathcal{Q}_s(n)$ and intersecting antichains in $\mathcal{A}_I(n)$ are both subsets of $\mathcal{P}(n)\setminus \{\emptyset\}$. Since two elements of $\mathcal{P}(n)\setminus \{\emptyset\}$ are in a 3-set $\{X,Y,X\cup Y\}$ for nonempty disjoint $X,Y\subseteq [n]$ if and only if they are comparable or disjoint, it immediately follows that a subset of $\mathcal{P}(n)\setminus \{\emptyset\}$ is a scarce linear set if and only if it is an intersecting antichain.
\end{proof}

At this point, our goal is to determine $|\mathcal{A}_I(n)|$, the number of intersecting antichains in the Boolean lattice $\mathcal{P}(n)\setminus \{\emptyset\}$.

\begin{lemma}
\label{lem:lower_bound}
    $|\mathcal{A}_I(n)| \ge 2^{{n \choose \lceil (n+1)/2 \rceil}}$.
\end{lemma}

\begin{proof}
Let $A$ be the collection of subsets of $[n]$ that have size $\lfloor \frac{n}{2} \rfloor + 1$. 
There are ${n \choose \lfloor n/2 \rfloor + 1}$ such subsets, hence $A$ has size ${n \choose \lfloor n/2 \rfloor + 1}$.

Since all sets in $A$ have the same size, it follows that $A$ is an antichain of $\mathcal{P}(n)$. Furthermore, since $\lfloor \frac{n}{2} \rfloor + 1 > \frac{n}{2}$, each pair of sets in $A$ intersect in a nonempty set. This implies that $A$ is an intersecting antichain of $\mathcal{P}(n)$.

Notice that each subset of $A$ is also an intersecting antichain. Therefore, there are at least $2^{{n \choose \lfloor n/2 \rfloor + 1}}$ intersecting antichains. Since $\lfloor \frac{n}{2} \rfloor + 1 = \lceil \frac{n+1}{2} \rceil$, the result follows.
\end{proof}

Antichains in the Boolean lattice are well studied and each intersecting antichain is also an antichain. Thus, we use the following theorem of Kleitman to determine an upper bound on the number of intersecting antichains in $\mathcal{P}(n)\setminus \{\emptyset\}$.

\begin{thm}[\cite{kleitman}]
\label{thm:kleitman}
    The number of antichains in $\mathcal{P}(n)$ is $2^{{n \choose n/2}(1+\text{o}(1))}$.
\end{thm}

It follows from Theorem \ref{thm:kleitman} that the number of intersecting antichains in $\mathcal{P}(n)\setminus \{\emptyset\}$ is at most $2^{{n \choose n/2 }(1+\text{o}(1))}$.
We are now able to determine the number of scarce linear subsets of $\mathcal{P}(n)$.

\begin{lemma}
\label{lem:scarce_biases}
    $|\mathcal{Q}_s(n)| = 2^{{n \choose n/2}(1+o(1))}$.
\end{lemma}

\begin{proof}
Observe that ${n \choose \lceil (n+1)/2 \rceil} = {n \choose n/2}(1+o(1))$. 
Now it follows from Lemma \ref{lem:lower_bound} and Theorem \ref{thm:kleitman} that $|\mathcal{A}_I(n)| = 2^{{n \choose n/2}(1+o(1))}$.
Now the result follows from Lemma \ref{lem:stable_sets}.
\end{proof}

The last step needed to prove Theorem \ref{thm:ext_clique} is to show that $|\mathcal{Q}(n)|$ is not ``much'' larger than $|\mathcal{Q}_s(n)|$. To do this we use an approach similar to that used in \cite{peter_jorn}. The strategy is to map the linear sets to scarce linear sets and show that the number that map to a specific scarce linear set is ``small''.

Recall that a linear set $\mathcal{B}$ of $\mathcal{P}(n)$ is a subset of $\mathcal{P}(n)$ such that if $A,B \in \mathcal{B}$ are in a 3-set $\{X,Y,X\cup Y\}$ where $X$ and $Y$ are disjoint and nonempty, then the third set in the 3-set $\{X,Y,X\cup Y\}$ is in $\mathcal{B}$ as well. 
Define a triple of sets $\{X,Y,X\cup Y\}$ to be a \textit{related triple} if $X$ and $Y$ are disjoint nonempty subsets of $[n]$. 

\begin{lemma}
\label{lem:scarce_vs_not}
$|\mathcal{Q}(n)| \le |\mathcal{Q}_s(n)|\cdot 2^{2{n \choose \le n/3}}$.
\end{lemma}

\begin{proof}
For each set $X \in \mathcal{P}(n)$, let $X^c$ denote the set $[n]\setminus X$.
Let $\prec$ be a linear ordering of $\mathcal{P}(n)$ such that $X\prec Y$ if $\min\{|X|,|X^c|\}<\min\{|Y|,|Y^c|\}$ for $X,Y \in \mathcal{P}(n)$.

If $\{X_0,X_1,X_2\}$ is a related triple, then without loss of generality $X_2 = X_0\cup X_1$.
Therefore, $\{X_0,X_1,X_2^c\}$ is a tripartition of $[n]$. Now it follows that at least one of $X_0,X_1,X_2^c$ has size at most $n/3$. This implies that $|X_i|$ or $|X_i^c|$ is at most $n/3$ for at least one of $i\in \{0,1,2\}$.

Let $\mathcal{S}$ be the collection of sets $X$ in $\mathcal{P}(n)\setminus \{\emptyset\}$ with $\min\{|X|,|X^c|\}\le n/3$. Notice that
\begin{equation*}
    |\mathcal{S}| \le 2\sum_{k=0}^{\lfloor n/3 \rfloor} {n \choose k}.
\end{equation*}

Let $\Lambda$ be the collection of all related triples $\{X_0,X_1,X_2\}$ where we always assume that $X_0\prec X_1\prec X_2$. By the discussion above, we know $X_0 \in \mathcal{S}$.
For each linear set $\mathcal{B} \subseteq \mathcal{P}(n)$, let $\phi(\mathcal{B})$ be obtained from $\mathcal{B}$ by simultaneously removing $X_0$ and $X_2$ for each triple $\{X_0,X_1,X_2\}\in \Lambda$ for which $\{X_0,X_1,X_2\} \subseteq \mathcal{B}$.

For all $\{X_0,X_1,X_2\}\in \Lambda$, since $|\{X_0,X_1,X_2\} \cap \mathcal{B}| \in \{0,1,3\}$, it follows that $|\{X_0,X_1,X_2\} \cap \mathcal{\phi(B)}|\le 1$; hence $\phi(\mathcal{B})$ is a scarce linear set.

\begin{claim}
For each scarce linear set $\mathcal{B}'\subseteq \mathcal{P}(n)$ and each set $\mathcal{X}\subseteq \mathcal{S}$, there exists at most one linear set $\mathcal{B}$ for which $\phi(\mathcal{B}) = \mathcal{B}'$ and $\mathcal{B}\cap \mathcal{S} = \mathcal{X}$.
\end{claim}

\begin{proof}
Suppose not. Thus, there exists a scarce linear set $\mathcal{B}'\subseteq \mathcal{P}(n)$ and distinct linear sets $\mathcal{B}_1,\mathcal{B}_2$ for which $\phi(\mathcal{B}_1) = \phi(\mathcal{B}_2) = \mathcal{B}'$ while $\mathcal{B}_1\cap \mathcal{S} = \mathcal{B}_2 \cap \mathcal{S}$.

Let $\mathcal{S}'$ be a maximal initial segment of $\mathcal{P}(n)\setminus \{\emptyset\}$ with respect to $\prec$ for which $\mathcal{B}_1\cap \mathcal{S}' = \mathcal{B}_2 \cap \mathcal{S}'$. We have $\mathcal{S} \subseteq \mathcal{S}' \neq \mathcal{P}(n)\setminus \{\emptyset\}$ by assumption.

Let $S \in (\mathcal{P}(n)\setminus \{\emptyset\}) \setminus \mathcal{S}'$ be minimal with respect to $\prec$. The maximality of $\mathcal{S}'$ implies that $S$ is in exactly one of $\mathcal{B}_1,\mathcal{B}_2$. Without loss of generality, say $S \in \mathcal{B}_1 \setminus \mathcal{B}_2$.

Since $\mathcal{B}' = \phi(\mathcal{B}_2) \subseteq\mathcal{B}_2$ and $S\notin \mathcal{B}_2$, we know that $S\notin \mathcal{B}' = \phi(\mathcal{B}_1)$. Since $S \in \mathcal{B}_1$, there exists a triple $\{X_0,X_1,X_2\}\in \Lambda$ where $\{X_0,X_1,X_2\}\subseteq \mathcal{B}_1$ and $S \in \{X_0,X_2\}$. Since $X_0 \in \mathcal{S}$ and $S\notin \mathcal{S}$, we know that $S = X_2$. 
By minimality of $S$ in $(\mathcal{P}(n)\setminus \{\emptyset\})\setminus \mathcal{S}'$ and since $X_0 \prec X_1\prec X_2$, it follows that $\{X_0,X_1\} \subseteq \mathcal{S}'$; hence $\{X_0,X_1\}\cap \mathcal{B}_2 = \{X_0,X_1\}\cap \mathcal{B}_1 = \{X_0,X_1\}$.
Now it follows that $X_0,X_1 \in \mathcal{B}_2$ and $|\{X_0,X_1,X_2\} \cap \mathcal{B}_2| = 2$, which is a contradiction.
\end{proof}

From the claim, we have that \begin{equation*}
    |\mathcal{Q}(n)| \le |\mathcal{Q}_s(n)|\cdot 2^{|\mathcal{S}|} \le |\mathcal{Q}_s(n)|\cdot 2^{2{n \choose \le n/3}}.\qedhere
\end{equation*}
\end{proof}

\vspace{2mm}
We are now ready to prove Theorem \ref{thm:ext_clique}.

\begin{proof}[Proof of Theorem \ref{thm:ext_clique}.]
Recall that extensions of $M(K_{n+1})$ correspond to linear sets by Lemma \ref{lem:bijection}.
The number of linear subsets of $\mathcal{P}(n)$ is at least the number of scarce linear subsets of $\mathcal{P}(n)$; therefore, by Theorem \ref{lem:scarce_biases} we find that the number of extensions is at least $2^{{n \choose n/2}(1+o(1))}$.
By Lemma \ref{lem:scarce_vs_not}, the number of extensions is at most
\begin{equation*}
    2^{{n \choose n/2}(1+o(1))} \cdot 2^{2{n \choose \le n/3}}.
\end{equation*}
Since $2{n \choose \le n/3} / {n \choose n/2} = o(1)$, it follows that the number of extensions is at most $2^{{n \choose n/2}(1+o(1))}$.
Combining the upper and lower bounds, we find that the number of extensions of $M(K_{n+1})$ is $2^{{n\choose n/2}(1+o(1))}$.
\end{proof}

\end{document}